\documentclass[11pt]{amsart}

\usepackage{url}
\usepackage{color}
\usepackage{graphicx}
\usepackage{amssymb}
\usepackage[colorlinks=true]{hyperref}
\usepackage{rotating}
\usepackage[lofdepth,lotdepth]{subfig}
\newcommand{\rotxc}[1]{\begin{sideways}#1\end{sideways}}
\newcommand{\invert}[1]{\rotxc{\rotxc{#1}}}

\newcommand{\0}{\mathbf{0}}
\renewcommand{\a}{\mathbf{a}}
\renewcommand{\b}{\mathbf{b}}

\newcommand{\C}{\mathcal{C}}

\newcommand{\FF}{\mathbb{F}}

\newcommand{\RR}{\mathbb{R}}
\newcommand{\QQ}{\mathbb{Q}}
\newcommand{\R}{\mathbb{R}}

\newcommand{\B}{\mathcal{B}}

\def\M{\mathcal{M}}
\def\O{\mathcal{O}}
\def\Le{\hbox{\invert{$\Gamma$}}}

\font\co=lcircle10

\def\jr{\rotatedown{\smash{\raise2pt\hbox{\co \rlap{\rlap{\char'005} \char'007}}
               \raise6pt\hbox{\rlap{\vrule height6.5pt}}
                \raise2pt\hbox{\rlap{\hskip4pt \vrule
          height0.4pt depth0pt
                width7.7pt}}}}}
\def\textcross{\ \smash{\lower4pt\hbox{\rlap{\hskip4.15pt\vrule height14pt}}
                \raise2.8pt\hbox{\rlap{\hskip-3pt \vrule height.4pt depth0pt
                                width14.7pt}}}\hskip12.7pt}

\def\textelbow{\ \hskip.1pt\smash{\raise2.75pt%
                \hbox{\co \hskip 4.15pt\rlap{\rlap{\char'004} \char'006}
                \lower6.8pt\rlap{\vrule height3.5pt}
                \raise3.6pt\rlap{\vrule height3.5pt}}
                \raise2.8pt\hbox{%
                  \rlap{\hskip-7.15pt \vrule height.4pt depth0pt
width3.5pt}%
                  \rlap{\hskip4.05pt \vrule height.4pt depth0pt
width3.5pt}}}
                \hskip8.7pt}

\DeclareMathOperator{\convex}{convex}
\DeclareMathOperator{\MacP}{MacP}
\DeclareMathOperator{\Gr}{Gr}
\DeclareMathOperator{\sign}{sign}

\newtheorem{theorem}{Theorem}[section]
\newtheorem{lemma}[theorem]{Lemma}
\newtheorem{proposition}[theorem]{Proposition}

\newtheorem{conjecture}[theorem]{Conjecture}
\theoremstyle{definition}
\newtheorem{definition}[theorem]{Definition}
\newtheorem{example*}[theorem]{Example}

\theoremstyle{remark}
\newtheorem{remark}[theorem]{Remark}

\DeclareRobustCommand{\qedify}[1]{%
  \ifmmode \quad\hbox{#1}
  \else
    \leavevmode\unskip\penalty9999 \hbox{}\nobreak\hfill
    \quad\hbox{#1}%
  \fi
}
\newenvironment{example}{\begin{example*}\pushQED{\qedify{$\diamondsuit$}}}{\popQED\end{example*}}

\makeatletter
\newtheorem*{rep@theorem}{\rep@title}
\newcommand{\newreptheorem}[2]{%
\newenvironment{rep#1}[1]{%
 \def\rep@title{{\bf #2 \ref{##1}}}%
 \begin{rep@theorem}}%
 {\end{rep@theorem}}}
\makeatother

\newreptheorem{theorem}{Theorem}

\begin{document}

\begin{abstract}
We prove da Silva's 1987 conjecture that any positively oriented matroid is a positroid; that is, it can be realized by a set of vectors in a real vector space.
It follows from this result and a result of the third author that 
the positive matroid Grassmannian (or \emph{positive MacPhersonian}) is homeomorphic to a 
closed ball.
\end{abstract}

\title{Positively oriented matroids are realizable}

\dedicatory{Dedicated to the memory of Michel Las Vergnas.}

\date{\today}
\thanks{
The first author was partially supported by the National Science Foundation CAREER Award DMS-0956178 and the SFSU-Colombia Combinatorics Initiative. 
The second author was supported by the EPSRC grant EP/I008071/1.  
The third author was partially supported by the National Science Foundation CAREER award
DMS-1049513.
}
\author{Federico Ardila}
\address{Mathematics Department, San Francisco State University, United States.}
\email{federico@sfsu.edu}

\author{Felipe Rinc\'on}
\address{Mathematics Institute, University of Warwick, United Kingdom.}
\email{e.f.rincon@warwick.ac.uk}

\author{Lauren Williams}
\address{Mathematics Department, University of California, Berkeley, United States.}
\email{williams@math.berkeley.edu}

\maketitle


\section{Introduction}
\label{sec:intro}

\emph{Matroid theory} was introduced in the 1930s as a combinatorial model that keeps track of, and abstracts, the dependence relations among a set of vectors. It has become an extremely powerful model in many other contexts, but its connections to linear algebra are still the subject of very interesting research today. Not every matroid arises from linear algebra, and one of the early hopes in the area was to discover the ``missing axiom" which characterizes the matroids that can be realized by a set of vectors. 
It is now believed that this is not a reasonable goal \cite{MNW, Vamos}, or in V\'amos's words, that ``the missing axiom of matroid theory is lost forever". 

While the realizability of a matroid over fields of characteristic zero is a very hard problem, the realizability over a finite field $\FF_q$ is more tractable. Geelen, Gerards, and Whittle recently announced a proof of Rota's 1970 conjecture that for any finite field $\FF_q$, there are only finitely many obstructions (``excluded minors") to being realizable over $\FF_q$. In other words, for realizability over a finite field there is indeed a finite list of ``missing axioms of matroid theory".

In a different but related direction, \emph{oriented matroid theory} was introduced in the 1970s as a model for real hyperplane arrangements; or equivalently, for the dependence relations among a set of real vectors together with their signs. Again, the problem of characterizing which oriented matroids actually come from real hyperplane arrangements is intractable. Even for orientations of uniform matroids, there is no finite set of excluded minors for realizability \cite{BS} \cite[Theorem 8.3.5]{OM}.

The problem of (oriented) matroid realizability over the field $\QQ$ of rational numbers is particularly hard. Sturmfels proved \cite{Stu} that the existence of an algorithm for deciding if any given (oriented) matroid is realizable over $\QQ$ is equivalent to the existence of an algorithm for deciding the solvability of arbitrary Diophantine equations within the field of rational numbers. It is also equivalent to the existence of an algorithm that decides if a given lattice is isomorphic to the face lattice of a convex polytope in rational Euclidean space. Despite much interest, all of these problems remain open.

\emph{Positively oriented matroids} were introduced by Ilda da Silva in 1987. They are oriented matroids for which all bases have a positive orientation. The motivating example is the uniform positively oriented matroid $\C^{n,r}$, which is realized by the vertices of the cyclic polytope $C^{n,r}$ \cite{Bland, Lv}. Da Silva studied the combinatorial properties of positively oriented matroids, and proposed the following conjecture, which is the main result of this paper.

\begin{conjecture} \label{conj} (da Silva, 1987 \cite{daS})
Every positively oriented matroid is realizable.
\end{conjecture}

More recently,  Postnikov \cite{postnikov} introduced \emph{positroids} in his study of the \emph{totally nonnegative part of the Grassmannian}. They are the (unoriented) matroids that can be represented by a real matrix in which all maximal minors are nonnegative. He unveiled their elegant combinatorial structure, and showed they are in bijection with several interesting classes of combinatorial objects, including Grassmann necklaces, decorated permutations, $\Le$-diagrams,  and equivalence classes of plabic graphs. They have recently been found to have very interesting connections with cluster algebras \cite{Scott} and quantum field theory \cite{AH}.

Every positroid gives rise to a positively oriented matroid, and da Silva's Conjecture \ref{conj} is the converse statement. This is our main theorem.

\begin{reptheorem}{thm:main} 
Every positively oriented matroid is a positroid, and is therefore realizable over $\QQ$.
\end{reptheorem}

There is a natural partial order on oriented matroids called
\emph{specialization}.
In \cite{MacPherson}, motivated by his theory of 
\emph{combinatorial differential manifolds}, 
MacPherson introduced the 
\emph{matroid Grassmannian} (also called the \emph{MacPhersonian})
$\MacP(d,n)$, which is the poset of rank $d$ oriented matroids on $[n]$
ordered by specialization.  He showed that 
$\MacP(d,n)$ plays the same role for
matroid bundles as the ordinary Grassmannian plays for vector bundles,
and pointed out that the geometric realization of the order complex
$\|\MacP(d,n)\|$ of $\MacP(d,n)$ is 
homeomorphic to the real Grassmannian 
$\Gr(d,n)$ if $d$ equals 
$1, 2, n-2$, or $n-1$.  
``Otherwise, the topology of the matroid Grassmannian is mostly a mystery."

Since MacPherson's work, some progress on this question has been made,
most notably by Anderson \cite{Anderson}, who obtained results
on homotopy groups of the matroid Grassmannian, and by
Anderson and Davis \cite{AD}, who constructed maps between the
real Grassmannian and the matroid Grassmannian -- showing that philosophically,
there is a splitting of the map from topology to combinatorics --
and thereby
gained some understanding of the mod $2$ cohomology
of the matroid Grassmannian.  However, many open questions remain.

We define the \emph{positive matroid Grassmannian} or 
\emph{positive MacPhersonian} $\MacP^+(d,n)$
to be the poset of rank $d$ positively oriented matroids on $[n]$,
ordered by specialization.  By Theorem \ref{thm:main}, each 
positively oriented matroid can be realized by an element of 
the \emph{positive Grassmannian} $\Gr^+(d,n)$.  Combining this fact
with results of the third author \cite{Shelling}, we obtain the 
following result.

\begin{theorem}
The positive matroid Grassmannian
$\|\MacP^+(d,n)\|$ is {ho\-me\-o\-mor\-phic} to a closed ball.  
\end{theorem}

The structure of this paper is as follows. In Sections \ref{sec:prelim} and \ref{sec:positroids} we recall some basic definitions and facts about matroids and positroids, respectively. In Section \ref{sec:POMs} we introduce positively oriented matroids, and prove some preliminary results about them. In Section \ref{sec:theorem} we prove da Silva's conjecture that all positively oriented matroids are realizable. 
Finally, in Section \ref{sec:MacPhersonian}, we introduce the 
positive MacPhersonian, and show that it is homeomorphic to a closed ball.

\section{Matroids}
\label{sec:prelim}

A matroid is a combinatorial object that unifies several notions of independence. Among the many equivalent ways of defining a matroid we will adopt the point of view of bases, which is one of the most convenient for the study of positroids and matroid polytopes. We refer the reader to \cite{Oxley} for a more in-depth introduction to matroid theory.

\begin{definition}
A \emph{matroid} $M$ is a pair $(E, \B)$ consisting of a finite set $E$ and a nonempty collection of subsets $\B=\B(M)$ of $E$, called the \emph{bases} of $M$, which satisfy the \emph{basis exchange axiom}: 
\begin{itemize}
\item If $B_1, B_2 \in \B$ and $b_1 \in B_1 - B_2$, then there exists $b_2 \in B_2 - B_1$ such that $(B_1 -  b_1) \cup b_2 \in \B$.
\end{itemize}
\end{definition}

The set $E$ is called the \emph{ground set} of $M$; we also say that $M$ is a matroid on $E$. 
A subset $F \subseteq E$ is called \emph{independent} if it is contained in some basis. 
The maximal independent sets contained in a given set $A \subseteq E$ are called the \emph{bases of $A$}. 
They all have the same size, which is called the \emph{rank} $r_M(A)=r(A)$ of $A$. 
In particular, all the bases of $M$ have the same size, 
called the rank $r(M)$ of $M$.  A subset of $E$ that is not independent
is called \emph{dependent}. A \emph{circuit} is a minimal dependent
subset of $E$ -- that is, a dependent set whose proper subsets
are all independent.

\begin{example}
Let $A$ be a $d \times n$ matrix of rank $d$ with entries in a field $K$, and denote its columns by ${\bf a}_1, {\bf a}_2, \dotsc, {\bf a}_n \in K^d$. The subsets $B \subseteq [n]$ for which the columns $\{{\bf a}_i \mid i \in B \}$ form a linear basis for $K^d$ are the bases of a matroid $M(A)$ on the set $[n]$. Matroids arising in this way are called \emph{realizable}, and motivate much of the theory of matroids.
\end{example}

There are several natural operations on matroids.

\begin{definition}\label{def:sum}
Let $M$ be a matroid on $E$ and $N$ a matroid on $F$. The \emph{direct sum}
of matroids $M$ and $N$ is the matroid $M \oplus N$
whose underlying set is
the disjoint union of $E$ and $F$, and whose bases are the disjoint 
unions of a basis of $M$ with a basis of $N$.
\end{definition}

\begin{definition}
Given a matroid $M=(E,\B)$, the \emph{orthogonal} or \emph{dual matroid} $M^*=(E,\B^*)$ is the 
matroid on $E$ defined by 
$\B^* := \{E - B \mid B\in \B\}$.
\end{definition}

A \emph{cocircuit} of $M$ is a circuit of the dual matroid $M^*$.

\begin{definition}
Given a matroid $M=(E,\B)$ and a subset $S \subseteq E$, the \emph{restriction}
of $M$ to $S$, written $M|S$, is the matroid on the ground set $S$ whose independent sets 
are all independent sets of $M$ which are contained in $S$. Equivalently,
the set of bases of $M|S$ is
\[
\B(M|S) = \{B \cap S \, \mid \, B \in \B \text{ and } |B \cap S| \text{ is maximal among all }B \in \B \}. 
\]
\end{definition}

The dual operation of restriction is contraction.

\begin{definition}
Given a matroid $M=(E,\B)$ and a subset $T \subseteq E$, the \emph{contraction}
of $M$ by $T$, written $M/T$, is the matroid on the ground set $E - T$ whose bases
are the following:
\[
\B(M/T) = \{B - T \, \mid \, B \in \B \text{ and } |B \cap T| \text{ is maximal 
among all }B \in \B \}. 
\]
\end{definition}

\begin{proposition}\cite[Chapter 3.1, Exercise 1]{Oxley}\label{prop:drc}
If $M$ is a matroid on $E$ and $T \subseteq E$, then 
\[
(M/T)^* = M^*|(E - T).
\]
\end{proposition}

The following geometric representation of a matroid will be
useful in our study of positroids.

\begin{definition} 
Given a matroid $M=([n],\B)$, the \emph{(basis) matroid polytope} $\Gamma_M$ of $M$ is the convex hull of the indicator vectors of the bases of~$M$:
\[
\Gamma_M := \convex\{e_B \mid B \in \B\} \subset \RR^n,
\]
where $e_B := \sum_{i \in B} e_i$, and $\{e_1, \dotsc, e_n\}$ is the standard basis of $\RR^n$.
\end{definition}


\begin{definition}
A matroid which cannot be written as the direct sum of two nonempty
matroids is called \emph{connected}.  Any matroid $M$ can be written uniquely as a direct sum of connected
matroids, called its \emph{connected components}; let $c(M)$ denote the number of 
{connected components} of $M$.
\end{definition}

Taking duals distributes among direct sums, so a matroid $M$ is connected if and only if its dual matroid $M^*$ is connected.

\begin{proposition} \cite{Oxley}.
\label{prop:equiv}
Let $M$ be a matroid on $E$.  For two elements $a, b \in E$, we set
$a \sim b$ whenever there are bases $B_1, B_2$ of $M$ such that
$B_2 = (B_1 - a) \cup b$. Equivalently, $a \sim b$ if and only if there
is a circuit $C$ of $M$ containing both $a$ and $b$. 
The relation $\sim$ is an equivalence
relation, and the equivalence classes are precisely the connected components of $M$.
\end{proposition}

The following lemma is well-known and easy to check.

\begin{lemma}\label{lem:dimpoly}
Let $M$ be a matroid on the ground set $[n]$.
The dimension of the matroid polytope $\Gamma_M$ equals 
$n-c(M)$.
\end{lemma}

The following result is a restatement of the greedy algorithm for matroids.

\begin{proposition}\label{prop:face}\cite[Exercise 1.26]{Coxetermatroids}, \cite[Prop. 2]{AK}
Let $M$ be a matroid on $[n]$.  
Any face of the matroid polytope
$\Gamma_M$ is itself a matroid polytope. 
More specifically, for $w: \RR^n \to \R$ let $w_i = w(e_i)$; by linearity, these values determine $w$. Now consider the flag of sets 
\[
\emptyset = A_0 \subsetneq A_1 \subsetneq \cdots \subsetneq A_k = [n] 
\]
such that $w_a=w_b$ for $a,b \in A_i - A_{i-1}$, 
and $w_a > w_b$ for $a \in A_i-A_{i-1}$ and $b \in A_{i+1}-A_i$.
Then the face of $\Gamma_M$ maximizing the linear functional $w$ is the matroid polytope of the matroid 
\[
\bigoplus_{i=1}^{k} (M|A_i)/{A_{i-1}}.
\]
\end{proposition}

\section{Positroids}
\label{sec:positroids}

We now introduce  a special class of realizable matroids 
introduced by Postnikov in \cite{postnikov}. 
We also collect several foundational results on positroids, which come from
\cite{oh, postnikov, ARW}.

\begin{definition}\label{def:positroid}
Suppose $A$ is a $d \times n$ matrix of rank $d$ with real entries such that all its maximal minors are nonnegative. Such a matrix $A$ is called \emph{totally nonnegative}, 
and the realizable matroid $M(A)$ associated to it is called a \emph{positroid}.
In fact, it follows from the work of Postnikov that any positroid can be realized 
by a totally nonnegative matrix with entries in $\QQ$ \cite[Theorem 4.12]{postnikov}.
\end{definition}

\begin{remark}
We will often identify the ground set of a positroid with the set $[n]$, but more 
generally, the ground set of a positroid may be any finite set $E=\{e_1,\dots,e_n\}$,
endowed with a specified total order 
$e_1<\dots<e_n$. Note that the fact that a given matroid is a positroid is strongly dependent
on the total order of its ground set; in particular, being a positroid is not invariant
under matroid isomorphism.
\end{remark}

\begin{example}
To visualize positroids geometrically, it is instructive to analyze the cases $d=2,3$. Some of these examples will be well-known to the experts; 
for example, part of this discussion also appears in \cite{AH}. 
Let the columns of $A$ be $\a_1, \ldots, \a_n \in \R^d$. 

\noindent {\it Case $d=2$}: Since $\det(\a_i, \a_j)$ is the signed area of the parallelogram generated by $\a_i$ and $\a_j$, we have that $0^\circ \leq \angle(\a_i,\a_j) \leq 180^\circ$ for $i<j$. Therefore the vectors $\a_1, \a_2, \ldots, \a_n$ appear in counterclockwise order in a half-plane, as shown in  Figure \ref{fig:2dim}.

\begin{figure}[ht]
\begin{center}
\includegraphics[height=3.4cm]{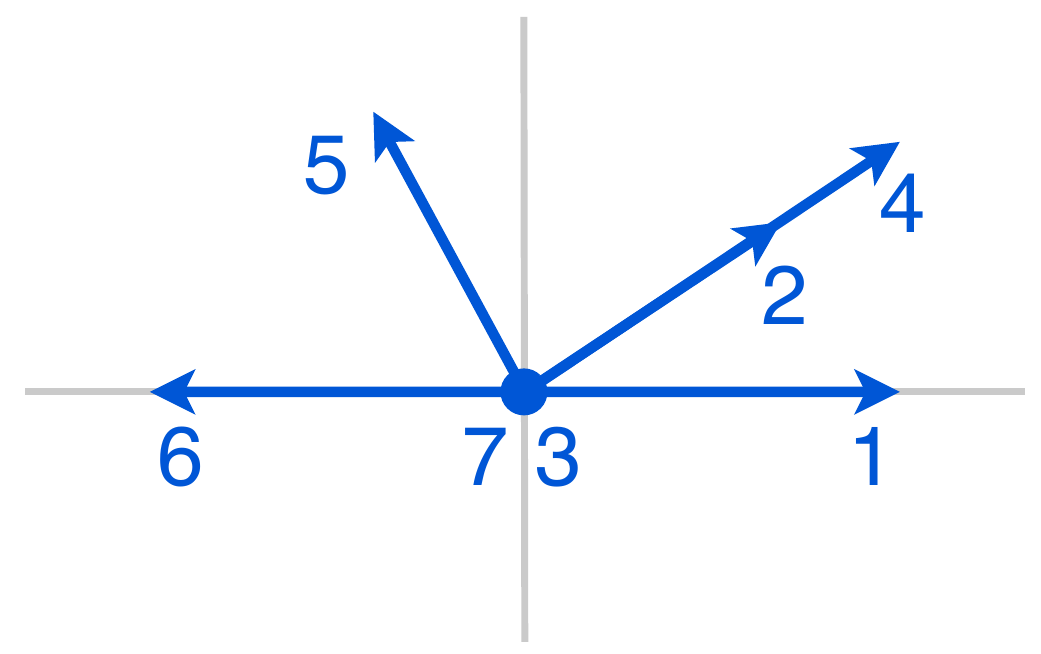}
 \caption{A realization of a positroid of rank $2$.}
 \label{fig:2dim}
\end{center}
\end{figure}

\medskip

\noindent {\it Case $d=3$}: Again we claim that $\a_1, \ldots, \a_n$ are contained in a half-space. If this were not the case, then the origin would be 
inside a triangular pyramid with affinely independent vertices 
$\a_{i_1}, \a_{i_2}, \a_{i_3}, \a_{i_4}$ for $i_1<\cdots<i_4$.
This would give 
$\lambda_1 \a_{i_1}+
\cdots +
\lambda_4 \a_{i_4}=\0$
for some $\lambda_1, \ldots, \lambda_4 > 0$. Then
\begin{eqnarray*}
0 &=& \det(\a_{i_1}, \a_{i_2},  \0) 
= \sum_{m=1}^4 \lambda_m \cdot \det(\a_{i_1},\a_{i_2}, \a_{i_m}) \\
&=&
\lambda_3 \cdot \det(\a_{i_1},\a_{i_2}, \a_{i_3}) +
\lambda_4 \cdot \det(\a_{i_1},\a_{i_2}, \a_{i_4}) > 0,
\end{eqnarray*}
a contradiction. 

There is no significant loss in assuming that our positroid contains no loops. Now there are two cases:

\begin{figure}[ht]
\begin{center}
\includegraphics[height=3.2cm]{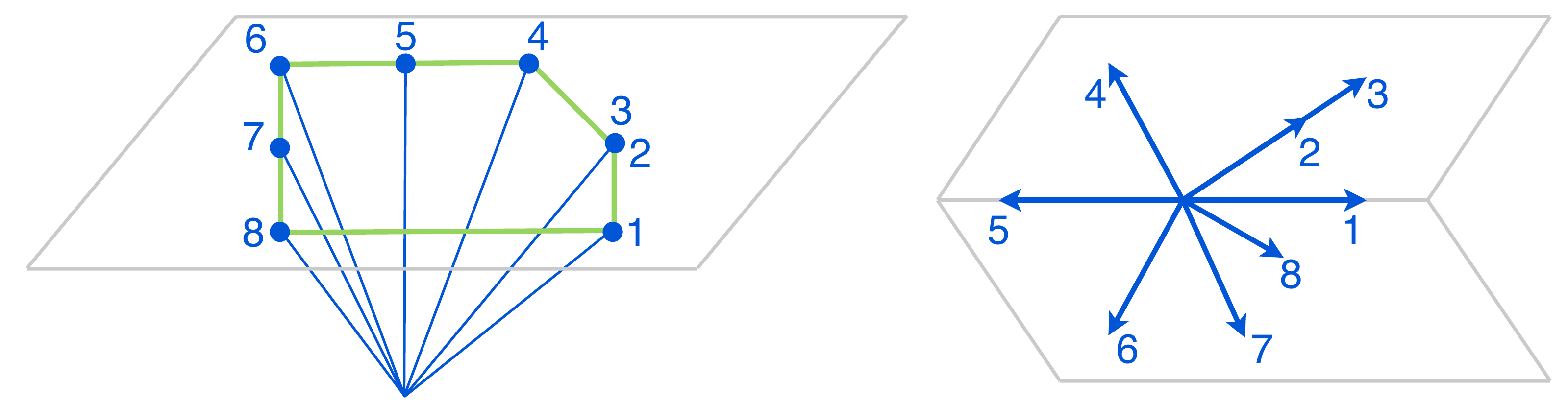}
 \caption{The two kinds of loop-free positroids of rank $3$.}
 \label{fig:3dim}
\end{center}
\end{figure}

\medskip

\noindent (a) The vectors $\a_1, \ldots, \a_n$ are in an \textbf{open} half-space.

After a suitable linear transformation and rescaling of the individual vectors, we may assume that $\a_i = [1, \b_i]^T$ for some row vector $\b_i \in \R^2$. Now $\det(\a_i,\a_j,\a_k)$ is the signed area of the triangle with vertices $\b_i, \b_j, \b_k$, so $\b_1, \ldots, \b_n$ must be the vertices (and possibly other points on the boundary) of a convex polygon, listed in counterclockwise order as shown in the left panel of Figure \ref{fig:3dim}.

\noindent (b) The vector $\0$ is in the convex hull of $\a_1, \ldots, \a_n$.

First assume that $\0 = \lambda_i \a_i + \lambda_j \a_j + \lambda_k \a_k$ where  $\lambda_i, \lambda_j, \lambda_k>0$ and $\a_i,\a_j,\a_k$ are affinely independent. Let $\a_l$ be one of the given vectors which is not on their plane. By \cite[Lemma 3.3]{ARW}, after possibly relabeling $i, j, k$, we may assume that $i<j<k<l$. 
This gives the following contradiction:
\[
0 = \det(\0,\a_k, \a_l) = \lambda_i \det(\a_i,\a_k, \a_l) + \lambda_j \det(\a_j,\a_k, \a_l)>0.
\]
Therefore $\0 = \lambda_i \a_i + \lambda_j \a_j$ for $\lambda_i, \lambda_j > 0$. If $\textrm{rank}(\{\a_i, \a_{i+1}, \ldots, \a_j\}) =3$, we would be able to find $i<r<s<j$ with
\[
0 = \det(\0, \a_r, \a_s) = \lambda_i \det(\a_i, \a_r, \a_s) + \lambda_j \det(\a_j, \a_r, \a_s) > 0.
\]
Thus $\textrm{rank}(\{\a_i, \a_{i+1}, \ldots, \a_j\})  \leq 2$ and similarly $\textrm{rank}(\{\a_j, \a_{j+1}, \ldots, \a_i\})  \leq 2$. Since our collection has rank $3$, these sets must both have rank exactly $2$. Hence our positroid is obtained by gluing the rank 2 positroids of $\a_i, \a_{i+1} \ldots, \a_j$ and $\a_j, \a_{j+1}, \ldots, \a_i$ along the line containing $\a_i$ and $\a_j$, as shown in the right panel of Figure \ref{fig:3dim}. One easily checks that this is a positroid when the angle from the second plane to the first is less than $180^\circ$.

\medskip

\noindent {\it Case $d>3$}: In higher rank, the idea that any basis among $\a_1, \ldots, \a_n$ must be ``positively oriented" is harder to visualize, and the combinatorics is now more intricate. However, we can still give a realization for the most generic positroid: it is given by any  points $f(x_1), \ldots, f(x_n)$  with $x_1 < \cdots < x_n$ on the moment curve $t \mapsto f(t) = (1, t, t^2, \ldots, t^{d-1})$ in $\R^d$. Every $d \times d$ minor of the resulting matrix is positive, thanks to the Vandermonde determinant. These $n$ points are the vertices of the \emph{cyclic polytope} $C^{n,r}$, whose combinatorics play a key role in the Upper Bound Theorem \cite{UBT}. 
In that sense, the combinatorics of positroids may be seen as a generalization of the combinatorics of cyclic polytopes.
\end{example}

If $A$ is as in Definition \ref{def:positroid} and $I \in \binom{[n]}{d}$ is a $d$-element subset
of $[n]$, then we let $\Delta_I(A)$ denote the $d \times d$ minor of $A$ 
indexed by the column set $I$.  These minors are called the
\emph{Pl\"ucker coordinates} of $A$.

In our study of positroids, we will repeatedly make use of the following notation. Given $k,\ell \in [n]$, we define the \emph{(cyclic) interval} $[k,\ell]$ to be the set
\[
[k,\ell] := 
\begin{cases}
 \{ k, k+1, \dotsc, \ell \} &\text{ if $k \leq \ell$},\\
 \{ k, k+1, \dotsc, n, 1, \dotsc, \ell \} &\text{ if $\ell < k$}.
\end{cases}
\]
We will often put a total order on a cyclic interval in the natural way.

The following proposition says that 
positroids are closed under duality, restriction, and contraction.
For a proof, see for example \cite{ARW}.

\begin{proposition} \label{prop:closed}
Let $M$ be a positroid on $[n]$.  Then $M^*$ is also a positroid on $[n]$.
Furthermore, for any subset $S$ of $[n]$, the restriction
$M|S$ is a positroid on $S$, 
and the contraction $M/S$ is a positroid on 
$[n]-S$.  Here the total orders 
on $S$ and $[n]-S$ are the ones
inherited from $[n]$.  
\end{proposition}

We say that two disjoint subsets $T$ and $T'$ of $[n]$ are \emph{non-crossing} if there is a cyclic interval of $[n]$ containing $T$ and disjoint from $T'$ (and vice versa). Equivalently, $T$ and $T'$ are non-crossing if there are no $a<b<c<d$ in cyclic order in $[n]$ such that $a, c \in T$ and $b,d \in T'$. 

If $S$ is a partition $[n]=S_1 \sqcup \cdots \sqcup S_t$ of $[n]$ into pairwise disjoint non-empty subsets, 
we say that $S$ is a \emph{non-crossing partition} if any two parts $S_i$ and $S_j$ are non-crossing. 
Equivalently, place the numbers $1,2,\dots, n$ on $n$ vertices around a circle in clockwise order, and then for each $S_i$ draw a polygon on the corresponding vertices. If no two of these polygons intersect, then $S$ is a non-crossing partition of $[n]$. 

Let $NC_n$ denote the set of non-crossing partitions of $[n]$.

\begin{theorem}\label{th:pos-sum}\cite[Theorem 7.6]{ARW}
Let $M$ be a positroid on $[n]$ and let $S_1$, $S_2$, \dots, $S_t$ be the ground sets of the connected components of $M$. Then $\Pi_M=\{S_1, \ldots, S_t\}$ is a non-crossing partition of $[n]$, called the \emph{non-crossing partition of $M$}.

Conversely, if $S_1$, $S_2$, \dots, $S_t$ form a non-crossing partition of $[n]$ and $M_1$, $M_2$, \dots, $M_t$ are connected positroids on $S_1$, $S_2$, \dots, $S_t$, respectively, then 
$M_1 \oplus \dots \oplus M_t$ is a positroid.
\end{theorem}

The following key result gives a 
characterization of positroids in terms of their matroid polytopes.

\begin{proposition}\label{prop:facets}\cite{LP}, \cite[Proposition 5.7]{ARW}
A matroid $M$ of rank $d$ on $[n]$ is a positroid if and only if its matroid polytope $\Gamma_M$ can be described by the equality $x_1+ \dotsb + x_n = d$ and inequalities of the form 
\[\sum_{\ell \in [i,j]} x_\ell \leq a_{ij}, \, \text{ with }i,j \in [n].\]
\end{proposition}

\section{Oriented matroids and positively oriented matroids}
\label{sec:POMs}

An oriented matroid is a signed version of the notion of matroid.  
Just as for matroids, there are several equivalent points of view and axiom systems.  
We will mostly focus on the chirotope point of view, but we will also use
the signed circuit axioms. For a thorough introduction to the theory of oriented matroids, see \cite{OM}.

\begin{definition}\cite[Theorem 3.6.2]{OM}\label{3term}
An \emph{oriented matroid} $\M$ of rank $d$ is a pair $(E,\chi)$ consisting of a finite set $E$
and a \emph{chirotope} $\chi: E^d \to \{-1,0,1\}$ that satisfies the following properties:
\begin{enumerate}
\item[(B$1'$)] The map $\chi$ is alternating, i.e.,
for any permutation $\sigma$ of $[d]$ and any $y_1,\dots,y_d \in E$, we have
 \[
  \chi(y_{\sigma(1)},\dots,y_{\sigma(d)}) = \sign(\sigma) \cdot \chi(y_1,\dots,y_d),
 \]
where $\sign(\sigma)$ is the sign of $\sigma$.
Moreover, the $d$-subsets $\{y_1, \dotsc, y_d\}$ of $E$ such that $\chi(y_1, \dotsc, y_d) \neq 0$ are the bases of a matroid on $E$.
\item[(B$2'''$)] For any $v_1, v_2, v_3, v_4, y_3,y_4, \dots,y_d \in E$,
\begin{center}
if $\epsilon:=\chi(v_1,v_2,y_3,y_4,\dots,y_d)  
\cdot \chi(v_3,v_4,y_3,y_4,\dots,y_d) \in \{-1,1\},$
\end{center}
then either
\begin{align*}
\chi(v_3,v_2,y_3,y_4,\dots,y_d)  
\cdot \chi(v_1,v_4,y_3,y_4,\dots,y_d) &= \epsilon \quad \text{ or}\\
\chi(v_2, v_4,y_3,y_4,\dots,y_d)  
\cdot \chi(v_1, v_3,y_3,y_4,\dots,y_d) &= \epsilon.
\end{align*}
\end{enumerate}
We consider $(E,\chi)$ to be the same oriented matroid
as $(E,-\chi)$.
\end{definition}

Definition \ref{3term} differs slightly from the usual definition
of chirotope, but it is equivalent to the usual definition
by \cite[Theorem 3.6.2]{OM}.  We prefer to work with the definition
above because it is closely related to the 3-term Grassmann-Pl\"ucker relations.

Note that the value of $\chi$ on a $d$-tuple
$(y_1,\dots,y_d)$ determines the value of $\chi$ on every $d$-tuple
obtained by permuting $y_1,\dots,y_d$.
Therefore when $E$ is a set with a total order
we will make the following convention:
if $I=\{i_1,\dots,i_d\}$ is a $d$-element subset of $E$ with 
$i_1 < \dots < i_d$ then 
we will let $\chi(I)$ denote
$\chi(i_1,\dots,i_d)$.  We may then think of 
$\chi$ as a function whose domain is the set of $d$-element subsets of $E$.

\begin{example}
Let $A$ be a $d \times n$ matrix of rank $d$ with entries in an ordered field $K$.
Recall that for a $d$-element subset $I$ of $[n]$ 
we let $\Delta_I(A)$ denote the determinant of the $d \times d$ submatrix of $A$
consisting of the columns indexed by $I$.
We obtain a 
chirotope $\chi_A: \binom{[n]}{d} \to \{-1,0,1\}$ by setting
\begin{equation}\label{eq:ch}
\chi_A(I) = 
\begin{cases}
0 &\mbox{if }\Delta_I(A)=0,\\
1 &\mbox{if }\Delta_I(A)>0,\\
-1 &\mbox{if }\Delta_I(A)<0.
\end{cases}
\end{equation}
An oriented matroid $\M=([n],\chi(A))$ arising in this way is called \emph{realizable over the field $K$}.
\end{example}

\begin{definition}
If $\M = (E,\chi)$ is an oriented matroid, 
its \emph{underlying matroid} $\underline{\M}$ is the (unoriented) matroid 
$\underline{\M}:=(E,\B)$ whose bases $\B$ are precisely the sets $\{b_1,\dots,b_d\}$ such that 
$\chi(b_1,\dots,b_d)$ is nonzero.
\end{definition}

\begin{remark}
Every oriented matroid $\M$ gives rise in this way to a matroid $\underline{\M}$.
However, given a matroid 
$(E,\B)$ it is not in general possible to give it the structure of an oriented matroid; that is,
it is  not always possible to find a chirotope
$\chi$ such that $\chi$ is nonzero precisely on 
the bases $\B$.
\end{remark}

\begin{definition}
 If $\M = (E,\chi)$ is an oriented matroid, any $A \subseteq E$ induces a \emph{reorientation} 
 $_{-A}\M := (E, \, _{-A}\chi)$ of $\M$, where $_{-A}\chi$ is the chirotope 
 \[
  _{-A}\chi(y_1,\dotsc,y_d) := (-1)^{|A\cap \{y_1, \dotsc,y_d\}|} \cdot \chi(y_1,\dotsc,y_d).
 \]
 This can be thought of as the oriented matroid obtained from $\M$ by ``changing the sign of the vectors in $A$''.
\end{definition}

The following definition introduces our main objects of study.

\begin{definition}
Let  $\M = (E,\chi)$ be an oriented matroid of rank $d$  on a set $E$
with a linear order $<$. We say $\M$ is  \emph{positively oriented with 
respect to $<$} if there is a reorientation $_{-A}\chi$ that makes all bases positive; that is,
\[
_{-A}\chi(I) \ := \ _{-A}\chi(i_1, i_2, \ldots, i_d) \ \geq \ 0
\]
for every $d$-element subset $I = \{i_1 < i_2 < \dotsc < i_d\} \subseteq E$.
\end{definition}

One can also define oriented matroids using the signed circuit axioms.

\begin{definition}
Let $E$ be a finite set.  Let $\mathcal{C}$ be a collection of \emph{signed
subsets} of $E$.
If $X\in \mathcal{C}$, we let $\underline{X}$ denote the underlying
(unsigned) subset of $E$, and $X^+$ and $X^-$ denote the subsets
of $\underline{X}$ consisting of the elements which have positive and negative signs,
respectively.  If the following
axioms hold for $\mathcal{C}$, then we say that $\mathcal{C}$
is the set of \emph{signed circuits} of an \emph{oriented matroid}
on $E$.
\begin{enumerate}
\item[(C0)] $\emptyset \notin \mathcal{C}$.
\item[(C1)] (symmetric) $\mathcal{C} = -\mathcal{C}$.
\item[(C2)] (incomparable) For all $X, Y \in \mathcal{C}$,
if  $\underline{X} \subset
\underline{Y}$, then $X=Y$ or $X=-Y$.
\item[(C3)] (weak elimination) for all $X,Y \in \mathcal{C}$,
$X \neq -Y$, and $e\in X^+ \cap Y^-$ there is a 
$Z\in \mathcal{C}$ such that 
\begin{itemize}
\item $Z^+ \subseteq (X^+ \cup Y^+) - e$ and 
\item $Z^- \subseteq (X^- \cup Y^-) - e$.
\end{itemize}
\end{enumerate}
\end{definition}

If $C$ is a signed subset of $E$ and $e \in C$, we will denote by $C(e)$ the sign of $e$ in $C$,
that is, $C(e) = 1$ if $e \in C^+$, and $C(e) = -1$ if $e \in C^-$. 

\begin{remark}\label{rem:equivalence}
The chirotope axioms and signed circuit axioms
for oriented matroids are equivalent. While the proof of this equivalence is intricate, the bijection is easy to describe, as follows.
For more details, see \cite[Theorem 3.5.5]{OM}. 
Given the chirotope $\chi$ of an oriented matroid $\M$, one can read off the bases of
the underlying matroid $\underline{\M}$ by 
looking at the subsets that $\chi$ assigns a nonzero value.
Then each circuit $\underline{C}$ of $\underline{\M}$ 
gives rise to a signed circuit $C$ (up to sign) as follows. 
If $e, f \in \underline{C}$ are distinct, let 
\[
\sigma(e,f) := - \chi(e,X)\cdot \chi(f,X) \in \{-1,1\},
\]
where $(f,X)$ is any ordered basis of $M$ containing $\underline{C}-e$.
The value of $\sigma(e,f)$ does not depend on the choice of $X$.
Let $c \in \underline{C}$, and let 
\begin{align*}
C^+ &:= \{ c\} \cup \{ f \in \underline{C}-c \mid \sigma(c,f) = 1\},\\
C^- &:= \{ f \in \underline{C}-c \mid \sigma(c,f) = -1\}.
\end{align*}
The signed circuit $C$ arising in this way does not depend (up to global sign) on the choice of $c$.
Finally, take $\C$ be the collection of signed circuits of $\M$ just described (together with their negatives). 
\end{remark}

\begin{lemma}\label{lemma:rotate}
If an oriented matroid $\M = ([n],\chi)$ is positively oriented with respect to the order $1 < 2 < \cdots < n$, then it is also positively oriented with respect to the  order $i < i+1 < \cdots < n < 1 < \cdots < i-1$, for any $1 \leq i \leq n$.
\end{lemma}

\begin{proof}
It suffices to prove this for $i=2$. After reorienting, we may assume that the bases of $\M$ are all positive with respect to the 
order $1<2<\dots<n$.
Consider a basis $B=\{b_1 < b_2 < \cdots < b_d\}$. If $1 \notin B$ then $B$ is automatically positive with respect to the new order. Otherwise, if $1 \in B$ then
\[
\chi(b_2, \ldots, b_d, 1) = (-1)^{d-1}\chi(1, b_2, \ldots, b_d) = (-1)^{d-1}.
\]
Hence if $d$ is odd, all bases of $\M$ are positive with respect to $2 < \cdots < n < 1$. If $d$ is even, all bases of the reorientation $_{-\{1\}}\M$ are positive with respect to $2 < \cdots < n < 1$. In either case, the desired result holds. 
\end{proof}

\begin{definition}
Let $\M = (E, \chi)$ be an oriented matroid of rank $d$, and let $A \subseteq E$.
Suppose that $E-A$ has rank $d'$, and choose $a_1, \ldots, a_{d-d'} \in A$
such that $(E-A) \cup \{a_1, \ldots, a_{d-d'}\}$ has rank $d$.
The \emph{deletion} $\M - A$, or \emph{restriction} $\M|(E-A)$, is the oriented matroid on $E - A$ with chirotope
\[
\chi_{\M - A}(b_1, \ldots, b_{d'}) := \chi_M(b_1, \ldots, b_{d'}, a_1, \ldots, a_{d-d'}),
\]
for $b_1, \ldots, b_{d'} \in E-A$. This oriented matroid is independent of 
$a_1, \ldots, a_{d-d'}$. 
\end{definition}

Positively oriented matroids are closed under restriction.

\begin{lemma}\label{lemma:restriction}
Let $\M$ be a positively oriented matroid on $[n]$. 
For any $S \subseteq [n]$, the restriction
$\M|S$ is positively oriented on $S$. Here the total order
on $S$ is inherited from the order $1 < \cdots < n$.  
\end{lemma}

\begin{proof}
It suffices to show that the deletion $\M - i$ 
of $i$ is positively oriented for any element $1 \leq i \leq n$. By Lemma \ref{lemma:rotate}, we may assume that $i=n$. We can also assume, after reorientation, that the bases of $\M$ are positive.

If $r(\M - n) = r(\M) =: d$, then the bases of $\M-n$ are also bases of $\M$, and they inherit their (positive) orientation from $\M$. Otherwise, if $r(\M - n) = d-1$, then each basis $\{a_1 < \cdots < a_{d-1}\}$ of $\M-n$ satisfies
\[
\chi_{\M-n}(a_1, \ldots, a_{d-1}) = \chi_\M(a_1, \ldots, a_{d-1}, n) = 1,
\]
since $\{a_1 < \cdots < a_{d-1}<n\}$ is a basis of $\M$. 
\end{proof}

\begin{definition}
Let $\M_1 = (E_1, \chi_1)$ and $\M_2 = (E_2, \chi_2)$ be oriented matroids on disjoint sets having ranks $d_1$ and $d_2$, respectively. 
The \emph{direct sum} 
$\M_1 \oplus \M_2$ is the oriented matroid on the set $E_1 \sqcup E_2$ whose chirotope $\chi$ is 
\[
 \chi (e_1,\dotsc,e_{d_1},f_1, \dotsc,f_{d_2}):= \chi_1(e_1, \dotsc,e_{d_1}) \cdot \chi_2 (f_1, \dotsc, f_{d_2}).
\]
The corresponding underlying matroids satisfy 
\[
 \underline{\M_1 \oplus \M_2} = \underline{\M_1} \oplus \underline{\M_2}.
\]
It is not hard to check that if $\C_1$ and $\C_2$ are the sets of signed circuits of the two oriented matroids $\M_1$ and $\M_2$, then $\C_1 \sqcup \C_2$ is the set of signed circuits of their direct sum $\M_1 \oplus \M_2$.
We say that an oriented matroid $\M$ is \emph{connected} if it cannot be decomposed 
as a direct sum of two oriented matroids on nonempty ground sets.
\end{definition}

\begin{proposition}\label{prop:connected}
 An oriented matroid $\M$ is connected if and only if its underlying matroid $\underline{\M}$ is connected.
\end{proposition}

\begin{proof}
It is clear that if $\underline{\M}$ is connected
then $\M$ is connected. 
Conversely, suppose that $\mathcal{M}$ is a connected oriented matroid.
We assume, for the sake of contradiction, that
$\underline{\M} = \underline{\M_1} \oplus \underline{\M_2}$ is the 
direct sum of two matroids on disjoint ground sets $E_1$ and $E_2$.
Let $\C$ and $\underline{\C}$ be the sets of signed 
and unsigned circuits of $\M$, respectively, and let $\underline{\C_1}$ 
and $\underline{\C_2}$ be the sets of (unsigned) circuits of $\underline{\M_1}$ 
and $\underline{\M_2}$. We have 
$\underline{\C} = \underline{\C_1} \sqcup \underline{\C_2}$.
For $i = 1, 2$, let $\C_i$  be the set of signed circuits obtained by 
giving each circuit in $\underline{\C_i}$ the signature that it has in $\C$. 
One easily checks that each $\C_i$ satisfies the signed circuit axioms, 
and hence it
defines an orientation $\M_i$ of the matroid $\underline{\M_i}$. 
We claim that $\M = \M_1 \oplus \M_2$. 

Since $\C_1$ and $\C_2$ determine $\chi_1$ and $\chi_2$ up to sign only, we need to show that there is a choice of signs that satisfies 
\begin{equation}\label{eq:signs}
\chi(A_1, A_2) = \chi_1(A_1)\cdot \chi_2(A_2)
\end{equation}
for any ordered bases $A_1, A_2$ of $\M_1, \M_2$. Here $(A_1, A_2)$ denotes the ordered basis of $\M$ where we list $A_1$ first and then $A_2$. 
Choose ordered bases $B_1$ and $B_2$ of $\M_1$ and $\M_2$. We may choose $\chi_1(B_1)$ and $\chi_2(B_2)$ so that (\ref{eq:signs}) holds for $B_1$ and $B_2$. 
Notice that (\ref{eq:signs}) will also hold for any reordering of $B_1$ and $B_2$. 

Now we prove that (\ref{eq:signs}) holds for any adjacent basis, which differs from $B_1 \sqcup B_2$ by a basis exchange; we may assume it is
$B_1' \sqcup B_2$ where $B_1' = (B_1 - e) \cup f$ and $e$ is the first element of $B_1$.
If $B_1 = (e, e_2,\dotsc, e_m)$, order the elements of $B'_1$ as $B_1' = (f,e_2,\dotsc,e_m)$. 
If $C$ is the signed circuit (of $\M$ and $\M_1$) contained in $B_1 \cup f$, the pivoting property \cite[Definition 3.5.1]{OM} applied to $\M$ and $\M_1$ gives
\[
\chi(B_1', B_2) = - C(e)C(f) \chi(B_1, B_2), \qquad 
\chi_1(B_1') =  - C(e)C(f) \chi_1(B_1).
\]
Therefore, if (\ref{eq:signs}) holds for $B_1 \sqcup B_2$, it also holds for the adjacent basis $B_1' \sqcup B_2$. 
Since all bases of $\M$ are connected by basis exchanges, 
Equation (\ref{eq:signs}) holds for all bases.  Therefore
$\M = \M_1 \oplus \M_2$ as oriented matroids, which contradicts the connectedness of $\M$.
\end{proof}

\section{Every positively oriented matroid is realizable}
\label{sec:theorem}

The main result of this paper is the following.
\begin{theorem}\label{thm:main}
Every positively oriented matroid is realizable over $\QQ$. \linebreak
Equivalently, the underlying matroid of any positively oriented matroid is a positroid.
\end{theorem}

In the proof of Theorem \ref{thm:main} we will make use of the forward direction in the following characterization. 
The full result, due to da Silva, appears in the unpublished work \cite{daS}. 
For completeness, we include a proof of the direction we use.

\begin{theorem}[{\cite[Chapter 4, Theorem 1.1]{daS}}]\label{thm:ilda}
 A matroid $M$ on the set $[n]$ is the underlying matroid of a positively oriented matroid
 if and only if 
 \begin{itemize}
  \item for any circuit $C$ and any cocircuit $C^*$ satisfying $C \cap C^* = \emptyset$, \\
  the sets $C$ and $C^*$ are non-crossing subsets of $[n]$.
 \end{itemize}
\end{theorem}
\begin{proof}[Proof of the forward direction]
 Suppose $M$ is the underlying matroid of a positively oriented matroid $\mathcal{M} = ([n], \chi)$.
 After reorienting, we can assume that $\chi (B) = 1$ for any basis $B$ of $M$.
 Let $C$ be a circuit of $M$ and $C^*$ be a cocircuit of $M$ such that $C \cap C^* = \emptyset$.
 If $C$ and $C^*$ are not non-crossing subsets of $[n]$
 then there exist $a,b \in C$ and $x,y \in C^*$ such that 
 $1 \leq a < x < b < y \leq n$ or $1 \leq y < a < x < b \leq n$.
 
Consider the hyperplane $H = [n] - C^*$. Since $C$ is a circuit in the restriction $M|H$,
 there exist bases $A, B$ of $H$ such that $B = (A - a) \cup b$.
Let $r = |\{e \in A \ \vert \ e < x\}|$ and 
$s = |\{e \in B \ \vert \ e < x\}|$.  Clearly $r=s+1$.
Then $\chi(x,A)=(-1)^r = -(-1)^s = - \chi(x,B)$, where 
the elements of $A$ and $B$ are listed in increasing order.
 Similarly $\chi(y,A) = \chi(y,B)$. 
 However, this contradicts 
 the dual pivoting property (PV$^*$) of oriented matroids \cite[Definition 3.5.1]{OM},
which implies that $\chi(x,A)/\chi(y,A) = \chi(x,B)/\chi(y,B)$.
\end{proof}

Note that after Theorem \ref{thm:main} has been proved, 
the statement of Theorem \ref{thm:ilda} will also constitute a characterization of positroids.

\begin{remark}
 In \cite[Chapter 4, Definition 2.1]{daS}, da Silva studies the notion of ``circular matroids". 
 A rank $d$ matroid $M$ on $[n]$ is \emph{circular} if for any circuit $C$ of rank $r(C)<d$, 
 the flat $\overline{C}$ spanned by $C$ is a cyclic interval of $[n]$. As she observed, her Theorem \ref{thm:ilda} implies that every circular matroid is the underlying matroid of a positively oriented matroid. The converse statement was left open, and we now show that it is not true. 
 
 We will make use of the correspondence between positroids and (equivalence classes of) plabic graphs; for more information, see \cite{postnikov, ARW}.
 Consider the \emph{plabic graph} $G$ with \emph{perfect orientation} $\O$ depicted in Figure \ref{fig:plabic_notcircular}.
 Let $M$ be the corresponding 
 positroid on $[7]$. 
 Its bases are the 4-subsets $I \subseteq [7]$ for which there exists a flow from the source set $I_\O = \{1,2,4,5\}$ to $I$. 
 One easily verifies that $C = \{1,4,7\}$ is a circuit of rank $2$, and it is also a flat which is not a cyclic interval. Therefore $M$ is not circular. 
%
%
%
  \begin{figure}[h]
  \begin{center}
   \includegraphics[scale=0.55]{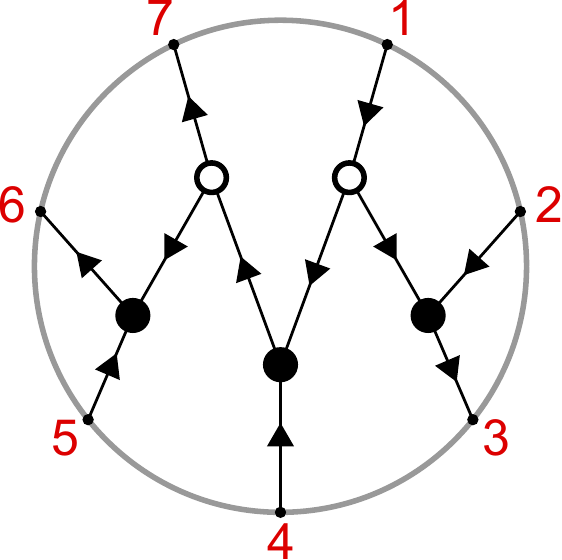}
   \caption{A perfect orientation of a plabic graph.}
   \label{fig:plabic_notcircular}
  \end{center}
 \end{figure}
\end{remark}

We now continue on our way toward proving Theorem \ref{thm:main}. 

\begin{proposition}\label{lem:POMdirectsum}
Let $\M$ be a positively oriented matroid on $[n]$ which is a direct sum 
of the connected oriented matroids $\M_1, \dots, \M_k$.  Let 
$S_1,\dots,S_k$ denote the ground sets of $\M_1,\dots,\M_k$.
Then $\M_1, \dots, \M_k$
are also positively oriented matroids, and 
$\{S_1,\dots,S_k\}$ is 
a non-crossing partition of $[n]$.
\end{proposition}

\begin{proof}
Each oriented matroid $\M_i = \M|S_i$ is positively oriented by Lemma \ref{lemma:restriction}. We need to prove that $S=\{S_1,\dots,S_k\}$ is 
a non-crossing partition of $[n]$. Consider any two distinct parts $S_i$ and $S_j$ of $S$.
By Proposition \ref{prop:connected}, the matroids
$\underline{\M_i}$ and $\underline{\M_j}$ are connected. It follows from 
Proposition \ref{prop:equiv} that if $a,b \in S_i$ then there is a circuit $C$ of $\underline{\M_i}$ 
(and thus a circuit of $\underline{\M}$)
containing both $a$ and $b$. Similarly, since the matroid $\underline{\M_j}$ is connected its dual
matroid $\underline{\M_j}^*$ is connected too, so for any $c,d \in S_j$ there is a cocircuit $C^*$ of $\underline{\M_j}$ 
(and thus a cocircuit of $\underline{\M}$)
containing both $c$ and $d$. The circuit $C$ and the cocircuit $C^*$ are disjoint, so by Theorem \ref{thm:ilda} they are non-crossing subsets of $[n]$.
The elements $a,b,c,d$ were arbitrary, so it follows that $S_i$ and $S_j$ are non-crossing, as desired.
\end{proof}

\begin{lemma}\label{lem:mainconnected}
If Theorem \ref{thm:main} holds for connected positively oriented matroids,
then it holds for arbitrary positively oriented matroids.
\end{lemma}
\begin{proof}
Let $\M$ be an arbitrary positively oriented matroid on $[n]$, and write it as
a direct sum of connected oriented matroids $\M_1,\dots, \M_k$ on the ground sets
$S_1,\dots,S_k$.
By Proposition \ref{lem:POMdirectsum}, each $\M_i$ is a positively oriented matroid, 
and $\{S_1,\dots,S_k\}$ is a non-crossing partition of $[n]$.
If Theorem \ref{thm:main} holds for connected positively oriented matroids then
each $\underline{\M_i}$ is a (connected) positroid.  But now by Theorem \ref{th:pos-sum},
their direct sum $\underline{\M}$ is a positroid.
\end{proof}

We now prove the main result of the paper.

\begin{proof}[Proof of Theorem \ref{thm:main}]
Let $\M$ be a positively oriented matroid of rank $d$ on $[n]$.
By Lemma \ref{lem:mainconnected}, we may assume that $\M$ is connected.
It follows from Proposition \ref{prop:connected} that its underlying matroid 
$M := \underline{\M}$ is connected.
By Lemma \ref{lem:dimpoly}, its matroid polytope $\Gamma_M$ has
dimension $\dim(\Gamma_M) = n-1$.
Moreover, any facet of 
$\Gamma_M$ is the matroid polytope of a matroid with exactly two connected components; so by Proposition \ref{prop:face}, 
it is the face of $\Gamma_M$ maximizing the dot product with a $0/1$-vector $w$.
Assume for the sake of contradiction that $M$ is not a positroid.  
It then follows from Proposition \ref{prop:facets} that
$\Gamma_M$ has a facet $F$ of the form 
$\sum_{i\in S} x_i = r_M(S)$, where $S \subseteq [n]$ is not a cyclic interval.
Each of the matroids $M|S$ and $M/S$ is connected.

Since $S$ is not a cyclic interval, we can find 
$i<j<k<\ell$ (in cyclic order) such that 
$i,k \in S$ and $j,\ell \notin S$.
In view of Proposition \ref{prop:equiv}, 
there exist bases $A \cup \{i\}$ and $A\cup \{k\}$ 
of $M|S$ exhibiting a basis exchange between $i$ and $k$.
Similarly, consider bases $B \cup \{j\}$ and $B \cup \{\ell\}$ of 
$M/S$ which exhibit a basis exchange between $j$ and $\ell$.
We now have the following bases of $M|S \oplus M/S$:
\begin{equation*}
A \cup B \cup \{i,j\},\quad A \cup B \cup \{i,\ell\},\quad
A \cup B \cup \{j,k\}, \quad A \cup B \cup \{k,\ell\}.
\end{equation*}
The corresponding vertices of $M$ are on $F$, so $w(e_{A \cup B \cup \{i,j\}}) = r(S)$. Then $A \cup B \cup \{i,k\}$ is \emph{not}
a basis of $M$, because 
$w(e_{A \cup B \cup \{i,k\}}) = w(e_{A \cup B \cup \{i,j\}}) + 1 = r(S)+1$,
since $i,k\in S$ and $j \notin S$. 

We now use Definition \ref{3term}.
Denote the elements of $A \cup B$ by $y_3, y_4, \dots, y_d$, 
where $y_3 < y_4 < \dots < y_d$.
We claim that 
\begin{equation}\label{claim}
\chi(i,j, y_3,\dots,y_d)  \chi(k,\ell, y_3, \dots, y_d) = 
\chi(j,k, y_3, \dots, y_d) 
\chi(i,\ell,y_3,\dots, y_d).  
\end{equation}
If we can prove the claim then we will contradict property (B$2'''$) of Definition \ref{3term}, because 
$\epsilon:= \chi(i,j, y_3,\dots,y_d)  \chi(k,\ell, y_3, \dots, y_d)$
is nonzero, but 
\begin{align*}
\chi(k,j, y_3,\dots,y_d)  \chi(i,\ell,y_3,\dots,y_d) &= 
- \chi(j,k, y_3,\dots,y_d) \chi(i,\ell,y_3,\dots,y_d)\\ &= 
- \chi(i,j, y_3,\dots,y_d)  \chi(k,\ell, y_3, \dots, y_d)\\ &= 
-\epsilon,
\end{align*}
and 
$\chi(i,k, y_3,\dots,y_d)  \chi(j,\ell,y_3,\dots, y_d) = 0$
since $A \cup B \cup \{i,k\}$ is not a basis.

Recall that if $I = \{ i_1 < \dots < i_d \}$,  we let 
$\chi(I) =  \chi(i_1, \dots, i_d)$.
Since $\M$ is  positively oriented, after reorienting we can assume
$\chi(I) \geq 0$ for all $d$-subsets $I$ of $[n]$.
We then have
\begin{equation*}
\chi(a,b,y_3,\dots,y_d) = (-1)^r
\chi(\{a\} \cup \{b\} \cup \{y_3, \dots, y_d\}) = (-1)^r,
\end{equation*}
where $r$ is the number of transpositions
needed to put the elements of the sequence
$(a,b,y_3,\dots,y_d)$ in increasing order.
Therefore to prove \eqref{claim}, we will compute $r$ for each 
term within it.

We know that $i<j<k<\ell$ in cyclic order.
In view of Lemma \ref{lemma:rotate}, we can assume that in fact $1 \leq i<j<k<\ell \leq n$. Define
\begin{align*}
c_1 &= | (A \cup B) \cap [1,i-1]|,\\
c_2 &= |(A \cup B) \cap [i+1,\dots,j-1]|,\\
c_3 & =|(A \cup B) \cap [j+1,\dots,k-1]|,\\
c_4 &= | (A \cup B) \cap [k+1,\dots,\ell-1]|.
\end{align*}
Then we have
\begin{align*}
\chi(i,j,y_3,\dots,y_d) & = (-1)^{c_1+c_1+c_2} = (-1)^{c_2}\\
\chi(k,\ell,y_3,\dots,y_d) & = (-1)^{2c_1+2c_2+2c_3+c_4} = (-1)^{c_4}\\
\chi(j,k,y_3,\dots,y_d) & = (-1)^{2c_1+2c_2+c_3} = (-1)^{c_3}\\
\chi(i,\ell,y_3,\dots,y_d) & = (-1)^{2c_1+c_2+c_3+c_4} = (-1)^{c_2+c_3+c_4}.
\end{align*}
Therefore
\[
\chi(i,j, y_3,\dots,y_d) \cdot \chi(k,\ell, y_3, \dots, y_d) = 
(-1)^{c_2+c_4},
\]
and also
\[
\chi(j,k, y_3, \dots, y_d) \cdot
\chi(i,\ell,y_3,\dots, y_d) = (-1)^{c_2+2c_3+c_4} = (-1)^{c_2+c_4},
\]
which proves the  claim.
\end{proof}

\section{The positive matroid Grassmannian is homeomorphic to a 
 ball}\label{sec:MacPhersonian}

In \cite{MacPherson}, MacPherson introduced the notion of \emph{combinatorial 
differential manifold}, a simplicial pseudomanifold with 
an additional discrete structure -- described in the language
of oriented matroids -- 
to model ``the tangent bundle."  He also developed the bundle theory
associated to combinatorial differential manifolds,
and showed that the classifying space of \emph{matroid bundles}
is the \emph{matroid Grassmannian} or \emph{MacPhersonian}.
The matroid Grassmannian therefore
plays the same role for 
matroid bundles as the ordinary Grassmannian plays for vector bundles.

After giving some preliminaries,
we will introduce the matroid Grassmannian and  
define its positive analogue.  The main result
of this section is 
that the positive matroid Grassmannian is homeomorphic
to a closed ball.

Given a poset, there is a natural topological object which one 
may associate to it, namely, the geometric realization of 
its order complex.
\begin{definition}
The \emph{order complex} $\|\mathcal{P}\|$ of a poset $\mathcal{P} = (P, \leq)$ is the simplicial
complex on the set $P$ whose simplices are the chains in $\mathcal{P}$.
\end{definition}

\begin{definition}
A CW complex is \emph{regular} if the closure $\overline{c}$ of each cell
$c$ is homeomorphic to a closed ball, and 
$\overline{c} \setminus c$ is homeomorphic to a sphere.
\end{definition}

Given a cell complex $\mathcal{K}$, we define its \emph{face poset}
$\mathcal{F}(\mathcal{K})$ to be the set of closed cells
ordered by containment, and augmented by a least element $\hat{0}$.
In general, the order complex 
$\|\mathcal{F}(\mathcal{K})-\hat{0}\|$ does not reveal the topology
of $\mathcal{K}$.  However, the following result shows that regular
CW complexes are combinatorial objects in the sense that the incidence
relations of cells determine their topology.

\begin{proposition}\cite[Proposition 4.7.8]{Bjorner}
Let $\mathcal{K}$ be a regular CW complex.  Then 
$\mathcal{K}$ is homeomorphic to $\|\mathcal{F}(\mathcal{K})-\hat{0}\|$.
\end{proposition}

There is a natural partial order on oriented matroids called
\emph{specialization}.

\begin{definition}
Suppose that $\M=(E,\chi)$ and $\M'=(E,\chi')$ 
are two rank $k$ oriented matroids on $E$.
We say that $\M'$ is a \emph{specialization} 
of $\M$, denoted $\M \leadsto \M'$, if  (after replacing $\chi$ with $-\chi$ if necessary)
we have that 
$$\chi(y_1,\dots,y_k)=\chi'(y_1,\dots,y_k) \text{ whenever }
\chi'(y_1,\dots,y_k) \neq 0.$$
\end{definition}

\begin{definition}
The \emph{matroid Grassmannian} or \emph{MacPhersonian}
$\MacP(k,n)$ of rank $k$ on $[n]$
is the poset of rank $k$ oriented matroids
on the set $[n]$, 
where 
$\M \geq \M'$ if and only if 
$\M \leadsto \M'$.
\end{definition}

One often identifies $\MacP(k,n)$ with its 
order complex.  When we speak of the topology 
of $\MacP(k,n)$, we mean the topology of 
(the geometric realization of) the order complex of 
$\MacP(k,n)$,  denoted
$\|\MacP(k,n)\|$.

MacPherson \cite{MacPherson} pointed out that
$\|\MacP(k,n)\|$ is 
homeomorphic to the real Grassmannian 
$\Gr(k,n)$ if $k$ equals 
$1, 2, n-2$, or $n-1$, but that 
``otherwise,
the topology of the matroid Grassmannian is mostly a mystery."
As mentioned in the introduction of this paper, Anderson \cite{Anderson},
and Anderson and Davis \cite{AD} made some progress on this question,
obtaining results on the homotopy groups and cohomology
of the matroid Grassmannian.
Shortly thereafter, the paper \cite{Biss} put forward a proof that
the matroid Grassmannian $\|\MacP(k,n)\|$ is homotopy equivalent
to the real Grassmannian $\Gr(k,n)$.  Unfortunately, a serious
mistake was found in the proof \cite{BissErratum}, and
it is still open whether
$\MacP(k,n)$ is homotopy equivalent to $\Gr(k,n)$.

We now introduce a positive counterpart $\MacP^+(k,n)$
of the matroid Grassmannian.
This space turns out to be more tractable than $\MacP(k,n)$;
we can 
completely describe its homeomorphism type.

\begin{definition}
The \emph{positive matroid Grassmannian} or \emph{positive MacPhersonian}
$\MacP^+(k,n)$ of rank $k$ on $[n]$
is the poset of rank $k$ positively oriented matroids
on the set $[n]$, 
where 
$\M \geq \M'$ if and only if 
$\M \leadsto \M'$.  
\end{definition}

For convenience, we usually augment $\MacP^+(k,n)$ by adding 
a least element $\hat{0}$.
Our main theorem on the topology of $\MacP^+(k,n)$ is the following.

\begin{theorem}\label{thm1}
$\MacP^+(k,n)$ 
is the face poset of a regular CW complex homeomorphic
to a ball. It follows that:
\begin{itemize}
\item $\|\MacP^+(k,n)\|$
is homeomorphic to a ball.
\item  
For each $\M \in \MacP^+(k,n)$,
the closed and open intervals
$\|[\hat{0},\M]\|$ and 
$\|(\hat{0},\M)\|$ are homeomorphic to 
a ball and a sphere, respectively.
\item 
$\MacP^+(k,n)$ is Eulerian.
\end{itemize}
\end{theorem}

The positive analogue of the real Grassmannian is the 
\emph{positive Grassmannian} (also called the 
\emph{totally non-negative Grassmannian}).  
The positive Grassmannian is an example of a 
positive flag variety, as 
introduced  by Lusztig in his theory of total positivity
for real flag manifolds \cite{Lusztig}, and its combinatorics
was beautifully developed by Postnikov \cite{postnikov}.
The positive Grassmannian has recently received a great
deal of attention because of its connection with 
\emph{scattering amplitudes} \cite{AH}.

\begin{definition}
The \emph{positive Grassmannian} $\Gr^+(k,n)$ is the 
subset of the real Grassmannian where all Pl\"ucker coordinates
are non-negative.
\end{definition}

While it remains unknown whether 
$\|\MacP(k,n)\|$ is homotopy-equivalent to $\Gr(k,n)$, 
the positive analogue of that statement is true.

\begin{theorem}\label{thm2}
The positive matroid Grassmannian $\|\MacP^+(k,n)\|$ and 
the positive Grassmannian $\Gr^+(k,n)$ are homotopy-equivalent; 
more specifically, both 
are contractible, with boundaries homotopy-equivalent to a sphere.
\end{theorem}

Before proving Theorems \ref{thm1} and \ref{thm2}, we review some 
results on the positive Grassmannian \cite{postnikov, Shelling, RW}.

Let $\B \subseteq \binom{[n]}{k}$ be a collection of $k$-element
subsets of $[n]$.  
We define
$$S_{\B}^{tnn} = \{A \in \Gr^+(k,n) \ \vert \ \Delta_I(A)>0
\text{ if and only if }I\in \B\}.$$

\begin{theorem} \cite{postnikov}\label{posetclosures}
Each subset $S_{\B}^{tnn}$ is either empty or a 
cell.  The positive Grassmannian $\Gr^+(k,n)$
is therefore a disjoint union of cells, where 
$S_{\B'}^{tnn} \subset \overline{S_{\B}^{tnn}}$ if and only if 
$\B' \subseteq \B$.
\end{theorem}

Let $Q(k,n)$ denote the poset of cells
of $\Gr^+(k,n)$, ordered by containment of closures, 
and augmented by a least element $\hat{0}$.

\begin{theorem}\cite{Shelling}\label{thm:shelling}
The poset $Q(k,n)$ is graded, thin, and EL-shellable.
It follows that $Q(k,n)$  is the face
poset of a regular CW complex homeomorphic to a ball, 
and that it is Eulerian.
\end{theorem}

\begin{theorem}\cite{RW}\label{thm:RW}
The positive Grassmannian
$\Gr^+(k,n)$ is {con\-tract\-i\-ble}, and its boundary is homotopy-equivalent
to a sphere.  Moreover, the closure of every cell is contractible,
and the boundary of every cell is homotopy-equivalent to a sphere.
\end{theorem}

\begin{remark}
In fact, Theorems \ref{thm:shelling} and \ref{thm:RW} were proved
more generally in \cite{Shelling, RW}
for real flag varieties $G/P$.
\end{remark}

We have the following result.

\begin{proposition}\label{equalposets}
For any $k \leq n$, 
$\MacP^+(k,n)$ and $Q(k,n)$ are isomorphic as posets.
\end{proposition}
\begin{proof}
By Theorem \ref{thm:main}, every positively oriented matroid is a
positroid.  Therefore each positively oriented matroid is 
realizable by a totally nonnegative matrix.  It follows from
the definitions that positively oriented matroids in 
$\MacP^+(k,n)$ are in bijection with the cells of 
$\Gr^+(k,n)$.  Moreover, by Theorem \ref{posetclosures},
the order relation (specialization)
in $\MacP^+(k,n)$ precisely corresponds to the order relation
on closures of cells in $\Gr^+(k,n)$.  
\end{proof}

Theorem \ref{thm1} now follows directly from Proposition 
\ref{equalposets}
and Theorem \ref{thm:shelling}, while
Theorem \ref{thm2} follows from Proposition
\ref{equalposets}
and Theorem \ref{thm:RW}.

\bibliographystyle{amsalpha}
\bibliography{bibliography}
\label{sec:biblio}

\end{document}